\documentclass{article}[12pt]

\usepackage[english]{babel}
\usepackage{amsmath,amssymb,hyperref,latexsym}

\newcommand{\mathd}{\mathrm{d}}
\newcommand{\nobracket}{}
\newcommand{\tmmathbf}[1]{\ensuremath{\boldsymbol{#1}}}
\newcommand{\tmop}[1]{\ensuremath{\operatorname{#1}}}
\newcommand{\tmtextit}[1]{\text{{\itshape{#1}}}}
\newenvironment{proof}{\noindent\textbf{Proof\ }}{\hspace*{\fill}$\Box$\medskip}
\newtheorem{lemma}{Lemma}
\newtheorem{theorem}{Theorem}

\begin{document}

\title{The dyadic Riesz vector I}

\author{Komla Domelevo, Stefanie Petermichl}

\maketitle

\begin{abstract}
  We derive a dyadic model operator for the Riesz vector. We show linear lower
  $L^p$ bounds for $1 < p < \infty$ between this model operator and the Riesz
  vector, when applied to functions with values in Banach spaces. By a lower
  bound we mean that the boundedness of the Riesz vector implies the
  boundedness of the dydic Riesz vector.
\end{abstract}

\section{Introduction}

Recently, a precise dyadic model was introduced by the authors, that allowed
for linear relations in the upper and lower $L^p$ estimates of the norms of
the Hilbert transform and its dyadic model operator with values in UMD Banach
spaces {\cite{DomPet2022a,DomPet2023a}}. The dyadic model for the Hilbert
transform was densely defined on the Haar system on $I_0 = [\nobracket 0, 1)
\nobracket$ by $\mathcal{S}_0 (\chi_{I_0}) = 0$, $\mathcal{S}_0 (h_{I_0}) = 0$
and for $I \in \mathcal{D} (I_0)$
\[ \mathcal{S}_0 : h_{I_{\pm}} \mapsto \pm h_{I_{\mp}} . \]
This dyadic operator shares many symmetries with the Hilbert transform: it is
antisymmetric, has a dyadic orthogonality and is an involution. We are now
interested in a model for the Riesz transforms. As we will explain below, the
natural choices for the dyadic Riesz transforms are sliced versions of
$\mathcal{S}_0$. If we are in dimension $d$, then we split the dyadic tree
into $d$ slices in the sense that the collection of all dyadic intervals
$\mathcal{D}$ in $[\nobracket 0, 1) \nobracket$ is split into $\mathcal{D}_m
(I_0) = \{ I \in \mathcal{D} (I_0) : | I | = 2^{- (l d + m)}, l \in \mathbb{N}
\}$ for $0 \leqslant m \leqslant d - 1$. The $k$th dyadic Riesz transform is
defined as $\mathcal{S}_k (\chi_{I_0}) = 0$ and acts non-trivially on Haar
functions on $\mathcal{D}_{k - 1}$:
\[ \mathcal{S}_k : \left\{ \begin{array}{lll}
     h_{\hat{I}_{\pm}} \mapsto \pm h_{\hat{I}_{\mp}} & : & I \in
     \mathcal{D}_{k - 1} (I_0)\\
     h_{\hat{I}_{\pm}} \mapsto 0 & : & I \in \mathcal{D} (I_0) \backslash
     \mathcal{D}_{k - 1} (I_0)
   \end{array} \right. \]
with the exception \ $\mathcal{S}_1 (h_{I_0}) = 0$. Notice that $\hat{I}$ is
the dyadic parent of $I$.

The choice is of course motivated by the estimates we obtain, but is also
very natural. One quite incomplete hint for this is the stochastic expression
that gives the Riesz transforms $R_k$ in $\mathbb{R}^d$ by Gundy--Varopoulos:
the martingale multiplier used in this representation is also slicing away
many increments:
\[ R_k (x) = c_d \lim_{y \rightarrow \infty} \mathbb{E} \left[ \int^{\tau}_0
   A_k \nabla P_t f (W_s^y) \mathd (W_s^y) | W^y_{\tau} = (0, x) \right], \]
where $W^y$ is the $d + 1$ dimensional Brownian motion in the upper half
space, stopped at time $\tau$ when hitting the boundary. Given the usual canonical 
base $\{e_0, e_1, \ldots e_k\} $, the matrix $A_k = e_k e^{\star}_0 - e_0
e^{\star}_k$ is very sparse, ignoring a number of increments of $\nabla P_t
f$.

The dyadic Riesz vector is then built in the usual way $\vec{\mathcal{S}} =
(\mathcal{S}_1, \ldots, \mathcal{S}_d)$. Again, we observe many similarities
with the continuous Riesz vector: antisymmetry, dyadic orthogonality,
involution.

In this note, we are interested when these operators act on UMD Banach space
valued functions. We will show that it is indeed possible to linearly dominate
the $L^p$ norm of this dyadic Riesz vector by that of the continuous Riesz
vector. In turn we also show the opposite direction in a separate text, namely
that the $L^p$ norm of the continuous Riesz vector can be linearly dominated
by the dyadic one.

\

These estimates are most interesting when the operators are applied to UMD
Banach space valued functions, but aspects of the precise estimates are also
very interesting in the real valued case: when estimating just one
$\mathcal{S}_k$ instead of the vector, the so obtained estimate in
proportional to the best estimates for the Hilbert transform, namely $\tan
(\pi / 2 p)$ for $1 < p \leqslant 2$ and $\cot (\pi / 2 p)$ for $2 < p <
\infty .$ Other texts that give estimates of the dyadic shifts are unable to
make use of orthogonality relations. On another note, the estimates are free
of dimension - this is an aspect absent in other texts, even when considering
dyadic shifts alone and not vectors thereof. A dimensional growth usually
translates into an increase of complexity of the shift operator and there is
at least a linear growth with complexity in the way these results are stated.
(We remark that in Treil's text {\cite{Tre2013a}}, the dimensional growth
arises from running the estimate for each slice, and our $\mathcal{S}_k$ are
sliced.)

Let us remind the reader that the `natural' representation of the continuous
Riesz vector by means of a stochastic integral does not give us dimension free
bounds. A minor (but very continuous trick) is required via the consideration
of another auxiliary stochastic integral to obtain dimensionless bounds (at a
cost of loss of orthogonality relations as well as a factor of 2). This trick
consisted of avoiding to collect $d$ terms in one row ($d$ being the dimension
of the input space) and by distributing these $d$ terms evenly at the cost of
a factor 2. Discrete estimates are almost always more difficult, so it is
remarkable that dimension free estimates can be made to carry over using the
methods we present in this paper, even though our focus lies on the linearity
of the norm relations.

Our method forces our hand to use purely odd dyadic operators that are time
faithful. This tends to be a class that can be more difficult to handle in
many circumstances because the operators are (in space) not local, sending
$h_{I_{\pm}} \mapsto \pm h_{I_{\mp}}$ instead of, say, $h_I \mapsto h_{I_+} -
h_{I_-}$ in case of the (local) classical shift (sometimes with a
normalization factor of $2^{- 1 / 2}$). Aside from these novelties in the real
valued case, our main result is of course the linear estimate for the vector
valued case, controlling the $L^p$ norm of $\vec{\mathcal{S}} =
(\mathcal{S}_1, \ldots, \mathcal{S}_d)$ by that of $\vec{R} = (R_1, \ldots,
R_d)$.

We make some remarks about the strategy and the novelty of techniques. Let us
first make some very easy, general remarks about the multipliers of Riesz
transforms and their actions on sinus and cosinus waves. The behavior on sinus
and cosinus waves is an important factor in our proof of the lower bound, that
works with sign tosses coded by \ $\tmop{sign} (\cos (\cdot))$ and
$\tmop{sign} (\sin (\cdot))$ square shaped waves that then relate back to the
continuous Riesz vector via a high oscillation argument. To see the lower
estimate, we take inspiration from Bourgain {\cite{Bou1983a}},
Geiss--Montgomery-Smith--Saksman {\cite{GeiMonSak2010a}} and the authors'
prior papers {\cite{DomPet2022a,DomPet2023a}}. If the reader already has read
these texts, then the following motivation will make sense (and otherwise it
likely will not before reading the proof).

The Riesz transforms $R_j$, $1 \leqslant j \leqslant d$ in $\mathbb{R}^d$
have Fourier multipliers $m_j = - i \frac{\xi_j}{| \xi |}, \xi \in
\mathbb{R}^d$. According to Geiss--Montgomery-Smith--Saksman, we know via
transference that the corresponding singular operator $\tilde{R}_j$ on
$\mathbb{T}^d$ with multiplier $\tilde{m}_j = - i \frac{n_j}{| n |}, n \in
\mathbb{Z}^d$ has the same norm, even if these operators act on functions with
values in Banach spaces. Starting out with dimension $d = 2$, we see that if
$j = 1$ and $n_2 = 0$ then the multiplier becomes $\tilde{m}_1 (n_1, 0) = - i
\frac{n_1}{| n |} = - i \tmop{sign} (n_1)$. In the Hilbert transform case, the
sign tosses on copies of $\mathbb{T}$ played a crucial role and the random
generators for the sign tosses were $\tmop{sqcos} (\theta) = \tmop{sign} (\cos
(\theta))$ and $\tmop{sqsin} (\theta) = \tmop{sign} (\sin (\theta))$. Now in
higher dimensions (say $d = 2$), our random generators sit on copies of
$\mathbb{T}^2$ taking the variable $\vec{\theta} = (\theta_1, \theta_2)$. The
function $\tmop{sqsin} \theta_1 = \tmop{sign} \sin \theta_1$ is constant in
variable $\theta_2$ and so its two dimensional Fourier transform is only
supported in $\mathbb{Z} \times \{ 0 \}$. One easily checks the action of
$\tilde{R}_j$ on the trigonometric functions $\sin \theta_1, \cos \theta_1,
\sin \theta_2, \cos \theta_2$. For example $\sin \theta_1$ as function on
$\mathbb{T}^2$ has Fourier support $\{ (- 1, 0), (1, 0) \}$ since $\sin
\theta_1 = \frac{1}{2 i} (e^{i \theta_1} - e^{- i \theta_1})$. Thus
$\tilde{R}_1 (\sin \theta_1) = - \cos \theta_1$ via the multipliers. Similarly
$\tilde{R}_1 (\cos \theta_1) = \sin \theta_1$, $\tilde{R}_2 (\cos \theta_1) =
0$, $\tilde{R}_2 (\sin \theta_1) = 0$. Another way to see this is to remark
that the Riesz transforms are directional Hilbert transforms. This observation
also motives the slicing of the dyadic shift. In the case of the Hilbert
transform in Bourgain and the author's recent text, there was a very direct
relation between the martingale multiplier or shift used by us and the
continuous operator. In the Riesz transform case, this is not quite so. In the
part of the argument using high oscillations, we find it necessary to
oscillate each component variable in each cluster, causing some interference,
where we no longer just have the signum multiplier of the Hilbert case. The
proof for the lower bound has to make use of estimates on the multiplier
kernel, a modification of a technique employed in \
Geiss--Montgomery-Smith--Saksman. \

\section{Main result}

Let $\mathcal{D}=\mathcal{D} (I_0)$ be the $L^2$ normalized dyadic system on
the unit interval $I_0$ and $\{ h^1_{I_0} \} \cup \{ h_I : I \in \mathcal{D}
\backslash \{ I_0 \} \}$ its Haar base. Recall that $\mathcal{S}_0$ is densely
defined by
\[ h^1_{I_0} \mapsto 0, h_{I_0} \mapsto 0, h_{I_{\pm}} \mapsto \pm h_{I_{\mp}}
\]
for $I \in \mathcal{D}$. Let us divide $\mathcal{D}$ into slices 0 through $d
- 1$: $\mathcal{D}_m = \{ I \in \mathcal{D}: | I | = 2^{- (l d + m)}, l \in
\mathbb{N}_0 \}$. Let for all $1 \leqslant j \leqslant d$ $\mathcal{S}_j$ be
densely defined by $\mathcal{S}_j : h_{I_0}^1 \mapsto 0$ and
\[ \mathcal{S}_j : \left\{ \begin{array}{lll}
     h_{\hat{I}_{\pm}} \mapsto \pm h_{\hat{I}_{\mp}} & : & I \in
     \mathcal{D}_{j - 1}\\
     h_{\hat{I}_{\pm}} \mapsto 0 & : & I \in \mathcal{D} \backslash
     \mathcal{D}_{j - 1}
   \end{array} \right. \]
admitting the exception $\mathcal{S}_1 : h_{I_0} \mapsto 0$. The dyadic Riesz
vector is $\vec{\mathcal{S}} = (\mathcal{S}_1, \ldots, \mathcal{S}_d)$. There
is a natural interpretation described below that places the dyadic system on a
cube of dimension $d$, where the dyadic increments have a directional instead
of a generational flavor and thus reminding us of the directional derivatives
in the continuous Riesz vector. The dyadic system with its slices is the most
concise way to state our main results. Let $1 < p < \infty$ and $X$ a UMD
space.

\begin{theorem}
  \label{th1}There exists a constant $c_0$ independent of $d$, $p$ and $X$
  such that
  \[ \| \vec{\mathcal{S}} \|_{L_X^p \rightarrow L_X^p (\ell^2)} \leqslant c_0
     \| \vec{R} \|_{L_X^p \rightarrow L_X^p (\ell^2)} . \]
\end{theorem}

\

\section{The lower bound strategy}

In this section, we will prove Theorem \ref{th1}.

\subsection{The dyadic Haar system on the cube}

In the last section, we defined the dyadic Riesz vector as having component
operators that acted on a slice in the dyadic Haar representation. It is easy
to see and well known that the dyadic tree also generates a Haar system in the
unit cube in $d$ dimensions.

Let us agree on the following $d$ step Haar system on the unit cube $Q_0$. We
start with
\[ h^1_{Q_0} = \chi_{Q_0} \frac{1}{\sqrt{| Q_0 |}} \]
and for all dyadic subcubes of $Q \subset Q_0$ define the $2^d$ dyadic
children $Q_{\sigma}$ where $\sigma = (\sigma_j)^d_{j = 1} \in \{ - 1, 1
\}^d$. Split cubes in a $d$ step process so that we have for $0 \leqslant k
\leqslant d - 1$
\[ Q_{\sigma_1 \ldots \sigma_k} = \prod_{j = 1}^k Q_{\sigma_j}^j \times
   \prod_{j = k + 1}^d Q^j \]
Then $Q_{\sigma_1 \ldots \sigma_d}$ is again a cube. The iterated Haar system
is
\[ h_{Q_{\sigma_1 \ldots \sigma_k}} = \frac{\chi_{Q_{\sigma_1 \ldots \sigma_k
   +}} - \chi_{Q_{\sigma_1 \ldots \sigma_k -}}}{\sqrt{| Q_{\sigma_1 \ldots
   \sigma_k} |}} . \]
In our indexing, we will not make this distinction, so we just write the
sequence counting the sign tosses from $Q_0$, keeping in mind that we run
through the $d$ dimensions in an ordered fashion.

\subsection{Coding}

It is well known that dyadic intervals can be seen as outcomes of a sequence
of sign tosses. We will choose random generators for these sign tosses that
are convenient for working with Riesz transforms. Let $\vec{\theta}^l =
(\theta_0^l, \ldots, \theta_{d - 1}^l) \in \mathbb{T}^d$. By the square sinus,
$\tmop{sqsin} \theta$, and square cosinus, $\tmop{sqcos} \theta$, we mean
$\tmop{sqsin} \theta = \tmop{sign} \circ \sin \theta$ and $\tmop{sqcos} \theta
= \tmop{sign} \circ \cos \theta$ respectively. Let us generate the sequences
of sign tosses.
\begin{eqnarray*}
  \varepsilon_0 & : & \varepsilon_0 = \tmop{sqcos} \theta^0_0\\
  \varepsilon_{l d + m}^-, \varepsilon_{l d + m}^+ & : & \varepsilon_{l d +
  m}^- = \tmop{sqsin} \theta^l_m, \varepsilon_{l d + m}^+ = \tmop{sqcos}
  \theta^l_m, \\
  && \forall l \geqslant 0, 0 \leqslant m \leqslant d - 1 : l d + m
  \neq 0
\end{eqnarray*}
The Haar expansion of $f$ is
\begin{eqnarray*}
  &  & \langle f \rangle_{Q_0}\\
  & + & \langle f, h_{Q_0} \rangle h_{Q_0}\\
  & + & \langle f, h_{Q_{0 -}} \rangle h_{Q_{0 -}} + \langle f, h_{Q_{0 +}}
  \rangle h_{Q_{0 +}}\\
  & + & \langle f, h_{Q_{0 - -}} \rangle h_{Q_{0 - -}} + \langle f, h_{Q_{0 -
  +}} \rangle h_{Q_{0 - +}} + \langle f, h_{Q_{0 + -}} \rangle h_{Q_{0 + -}} +
  \langle f, h_{Q_{0 + +}} \rangle h_{Q_{0 + +}}\\
  & + & \ldots
\end{eqnarray*}
where in each row only one summand is active, depending on the position of
$x$, due to disjoint supports of the Haar functions in each row. Pathwise
these are, again with just one summand active in each row, depending on the
outcome of the previous sign tosses ($\mathd f_k^{\mp}$ is active if the prior
sign toss from $\varepsilon^-_{k - 1}$ or $\varepsilon^+_{k - 1}$ was $\mp$)
\begin{eqnarray*}
  &  & \mathd f_{- 1}\\
  & + & \mathd f_0 \varepsilon_0\\
  & + & \mathd f^-_1 (\varepsilon_0) \varepsilon^-_1 + \mathd f^+_1
  (\varepsilon_0) \varepsilon^+_1\\
  & + & \mathd f^-_2 (\varepsilon_0, \varepsilon^-_1, \varepsilon^+_1)
  \varepsilon^-_2 + \mathd f^+_2 (\varepsilon_0, \varepsilon^-_1,
  \varepsilon^+_1) \varepsilon^+_2\\
  & + & \ldots
\end{eqnarray*}
Generally we have summands of the form
\[ \mathd f^{\pm}_k (\varepsilon_0, \varepsilon^-_1, \varepsilon^+_1,
   \varepsilon^-_2, \varepsilon^+_2, \ldots, \varepsilon^-_{k - 1},
   \varepsilon^+_{k - 1}) \varepsilon^{\pm}_k . \]
Here, blocks like $\varepsilon^-_{l d}, \varepsilon^+_{l d}, \ldots,
\varepsilon^-_{l d + (d - 1)}, \varepsilon^+_{l d + (d - 1)}$ are encoded by
one (the $l$th) copy of $\mathbb{T}^d$, where We count $\varepsilon_0, \ldots,
\varepsilon^-_{d - 1}, \varepsilon^+_{d - 1}$ as a result of the $0$th copy of
$\mathbb{T}^d$. This fact is reflected in the indexing below. A typical term
looks as follows, where $k \geqslant 0$ and $0 \leqslant m \leqslant d - 1$.
\begin{eqnarray*}
&& \mathd f^{\pm}_{k d + m} (\varepsilon_0, \varepsilon^-_1, \varepsilon^+_1,
   \varepsilon^-_2, \varepsilon^+_2, \ldots, \varepsilon^-_{k d + (m - 1)},
   \varepsilon^+_{k d + (m - 1)}) \varepsilon^{\pm}_{k d + m} \\
&=& \mathd F^{\pm}_{k d + m} (\vec{\theta}^0, \ldots, \vec{\theta}^k) \varphi^{\pm}(\theta_m^k) . 
\end{eqnarray*}   
To be precise, $\mathd F^{\pm}_{k d + m}$ only depends upon $(\theta_0^k,
\ldots, \theta_{m - 1}^k)$ in the last variable cluster $\vec{\theta}^k$.
Write
\[ \mathd F_{\tmmathbf{k}}^{\pm} (\vec{\theta}^0, \ldots, \vec{\theta}^k) =
   \sum_{m = 0}^{d - 1} \mathd F^{\pm}_{k d + m} (\vec{\theta}^0, \ldots,
   \vec{\theta}^k) \varphi^{\pm} (\theta_m^k) . \]
Assuming the Banach space valued functions are simple and of the form
$\sum^P_{p = 1} f_p x_p$ with $f_p$ smooth and scalar valued, we get
expressions $\mathd F^{\pm p}_{k d + m} $such as above for each $f_p$ that are
scalar valued.

\subsection{Preliminary key lemmata}\quad

\

The operators we claim model the Riesz transforms are simply sliced versions
of the original dyadic shift $\mathcal{S}_0 : h_{I_{\pm}} \mapsto \pm
h_{I_{\mp}}$. In the language of sign tosses, these are
\[ \mathcal{S}_1 : \left\{ \begin{array}{l}
     \varepsilon^{\pm}_{l d} \mapsto \pm \varepsilon^{\mp}_{l d}\\
     \varepsilon^{\pm}_{l d + 1} \mapsto 0\\
     \ldots\\
     \varepsilon^{\pm}_{l d + (d - 1)} \mapsto 0
   \end{array} \right., \mathcal{S}_2 : \left\{ \begin{array}{l}
     \varepsilon^{\pm}_{l d} \mapsto 0\\
     \varepsilon^{\pm}_{l d + 1} \mapsto \pm \varepsilon^{\mp}_{k d + 1}\\
     \ldots\\
     \varepsilon^{\pm}_{l d + (d - 1)} \mapsto 0
   \end{array} \right., \ldots, \mathcal{S}_d : \left\{ \begin{array}{l}
     \varepsilon^{\pm}_{l d} \mapsto 0\\
     \varepsilon^{\pm}_{l d + 1} \mapsto 0\\
     \ldots\\
     \varepsilon^{\pm}_{l d + (d - 1)} \mapsto \pm \varepsilon^{\mp}_{l d + (d
     - 1)}
   \end{array} \right. \]
with one exception when $l = 0$. Here, we recall that we conveined
$\mathcal{S}_1 : \varepsilon_0 \mapsto 0$. The dyadic Riesz vector is
$\vec{\mathcal{S}}$=$(\mathcal{S}_1, \ldots, \mathcal{S}_d)$. So
$\mathcal{S}_j : \varepsilon^{\pm}_{l d + j - 1} \mapsto \pm
\varepsilon^{\mp}_{l d + j - 1}$ and $\varepsilon^{\pm}_{l d + m} \mapsto 0$
if $m \neq j - 1$ with the convention $\mathcal{S}_1 : \varepsilon_0 \mapsto
0$.

\

Notice that $\tilde{R}_j$ on $\mathbb{T}^l$ does not see
$\varepsilon^{\pm}_{l d + m} = \varphi^{\pm} (\theta^l_m)$ for $m \neq j - 1$,
which will allow us to slice. In UMD spaces, slicing usually comes at a cost
of a copy of the UMD constant. We try to circumvent this dependence by the use
of the Riesz transform as dominating operator instead of the Hilbert
transform.

\

By density arguments mentioned above, we can write \
\begin{equation}
  \Phi^{\pm}_{\tmmathbf{k}} (\vec{\theta}^0, \ldots, \vec{\theta}^k) = \sum_{m
  = 0}^{d - 1} \Phi^{\pm}_{k d + m} (\vec{\theta}^0, \ldots, \vec{\theta}^k) =
  \sum_{p = 1}^P \sum^{d - 1}_{m = 0} \Phi^{\pm p}_{k d + m} (\vec{\theta}^0,
  \ldots, \vec{\theta}^k) x_p \label{denseelementk-1}
\end{equation}

where we use the notation
\[ \mathd F^{\pm p}_{k d + m} (\vec{\theta}^0, \ldots, \vec{\theta}^k)
   \varphi^{\pm} (\theta_m^k) = \Phi^{\pm p}_{k d + m} (\vec{\theta}^0,
   \ldots, \vec{\theta}^k) . \]
This allows us to have Fourier transforms such as
\[ \Phi^{\pm p}_{k d + m} (\vec{\theta}^0, \ldots, \vec{\theta}^k) =
   \sum_{\vec{l}^0 \in \mathbb{Z}^d} \ldots \sum_{\vec{l}^k \in \mathbb{Z}^d}
   e^{i \langle \vec{l}^0, \vec{\theta}^0 \rangle} \ldots e^{i \langle
   \vec{l}^k, \vec{\theta}^k \rangle} \alpha_{k d + m}^{\pm p} (\vec{l}^0,
   \ldots, \vec{l}^k) . \]
Recall that $\mathd F^{\pm p}_{k d + m} (\vec{\theta}^0, \ldots,
\vec{\theta}^k) \varphi^{\pm} (\theta_m^k)$ does not depend on $(\theta_{m +
1}^k, \ldots, \theta_{d - 1}^k)$, so that $\alpha_{k d + m}^{\pm p}
(\vec{l}^0, \ldots, \vec{l}^k) \neq 0 \Rightarrow l_{m + 1}^k = \cdots = l_{d
- 1}^k = 0$. Further, \ $\varphi^{\pm}$ has mean 0, so $\alpha_{k d + m}^{\pm
p} (\vec{l}^0, \ldots, (l_0^k, \ldots l_{m - 1}^k, 0, \ldots)) = 0$. To
summarize, $\alpha_{k d + m}^{\pm p}$ is supported where the last entry
$\vec{l}^k$ is of the form $l_m^k \neq 0$ and $l_{m + 1}^k = \cdots = l_{d -
1}^k = 0$.

\

By an approximation argument, we assume for now that the spectra are finite.
Let us denote by $E_{\tmmathbf{k}}$ the closure in $L_X^p$ for $p \in (0, 1)$
of expressions (\ref{denseelementk-1}) with the additional restrictions as
described above.

\

Consider the multiplier operators below inspired by the $j$th Riesz transform
as only acting on the last increment(s) in elements in $E_{\tmmathbf{k}}$. If
the last increment is $l_m^k$ (such as for $\Phi^{\pm}_{k d + m}$) with $l^k =
(l_0^k, \ldots, l_m^k, 0, \ldots 0)$ with $l_m^k \neq 0$ then use the
multiplier of the directional Hilbert transform
\[ \tilde{m}^k_j (\vec{l}^k) = \tilde{m}_j (0, \ldots, 0, l_m^k, 0, \ldots 0),
\]
where we recall $\tmop{sign} (0) = 0.$ Then define for $0 \leqslant m
\leqslant d - 1$
\begin{eqnarray*}
&& \widetilde{H_j}^k \Phi^{\pm}_{k d + m} (\vec{\theta}^0, \ldots,
   \vec{\theta}^k) \\
&=& \sum_{p = 1}^P \left( \sum_{\vec{l}^0 \in \mathbb{Z}^d}
   \ldots \sum_{\vec{l}^k \in \mathbb{Z}^d} \tilde{m}^k_j (\vec{l}^k) e^{i
   \langle \vec{l}^0, \vec{\theta}^0 \rangle} \ldots e^{i \langle \vec{l}^k,
   \vec{\theta}^k \rangle} \alpha_{k d + m}^{\pm p} (\vec{l}^0, \ldots,
   \vec{l}^k) \right) x_p . 
\end{eqnarray*}   
Notice that in order to have a contribution of $\widetilde{H_j}^k
\Phi^{\pm}_{k d + m}$, we need $m = j - 1$. Observe that this just means
moving the Riesz transform to the last increment:
\begin{eqnarray}
  \nonumber\widetilde{H_j}^k \Phi^{\pm}_{k d + m} (\vec{\theta}^0, \ldots,
  \vec{\theta}^k) 
  &=& \widetilde{H_j}^k \mathd F^{\pm}_{k d + m}
  (\vec{\theta}^0, \ldots, \vec{\theta}^k) \varphi^{\pm} (\theta_m^k)\\
  & =& \mathd
  F^{\pm}_{k d + m} (\vec{\theta}^0, \ldots, \vec{\theta}^k) \tilde{R}_j
  \varphi^{\pm} (\theta_m^k) . \label{insideRiesz}
\end{eqnarray}
To find the norm of the $d$ dimensional operator $\vec{\mathcal{S}}$, we will
dualize against a $d$ component test function $\vec{G} = (G_j)^d_{j = 1}$,
where each $G_j$ relates to $\Gamma_j$ as $F$ to $\Phi$.

\begin{lemma}
  \label{lmRmoved}Let $1 < p < \infty$ and $q$ its conjugate exponent. If
  $\Phi_{\tmmathbf{k}} \in E_{\tmmathbf{k}}$, $\Gamma_{j, \tmmathbf{r}} \in
  E_{\tmmathbf{r}}$ for all $1 \leqslant \tmmathbf{k}, \tmmathbf{r} \leqslant
  M, 1 \leqslant j \leqslant d$ one has
  \begin{eqnarray*}
    &  & \mathbb{E}^{\vec{\theta}} \sum_{j = 1}^d \left\langle \sum_{k = 1}^M
    \sum_{m = 0}^{d - 1} \sum_{\sigma} \widetilde{H_j}^k \Phi^{\sigma}_{k d +
    m} (\vec{\theta}^0, \ldots, \vec{\theta}^k), \right.\\
    && \hspace{4em}
    \left.
    \sum_{r = 1}^M \sum_{n =
    0}^{d - 1} \sum_{\rho} \Gamma^{\rho}_{j, r d + n} (\vec{\theta}^0,
    \vec{\theta}^1, \ldots, \vec{\theta}^r) \right\rangle_{X, X^{\star}}\\
    & \leqslant & \| \vec{R} \|_{L_X^p \rightarrow L_X^p (\ell^2)} \| F
    \|_{L_X^p} \| \vec{G} \|_{L_{X^{\star}}^q (\ell^2)}
  \end{eqnarray*}
\end{lemma}

\begin{proof}

  \
  
  By an approximation argument, we may think of all arising functions as
  bandlimited and simple. We modulate as follows: Let $\vec{\eta} \in
  \mathbb{Z}^d$ an auxiliary variable and $A$ a large integer, much larger
  than any of the spectra. Replace
  \[ (\theta_0^s, \ldots, \theta_{d - 1}^s) \longmapsto (\theta_0^s + A^{s d +
     1} \eta_0, \ldots, \theta_{d - 1}^s + A^{s d + d} \eta_{d - 1}) . \]
  If we apply the $j$th Riesz transform $\tilde{R}_j$ to the modulated
  $\Phi^{\pm p}_{k d + m}$ \ in the variable $\vec{\eta} = (\eta_0, \ldots,
  \eta_{d - 1})$, then the multiplier $\tilde{m}_j (n_0, \ldots, n_{d - 1}) =
  - i \frac{n_{j - 1}}{| n |}$ is evaluated at
  \[ (l_0^0 A + l_0^1 A^{d + 1} + \cdots + l_0^k A^{k d + 1}, \ldots, l_{d -
     1}^0 A^d + \cdots + l_{d - 1}^k A^{k d + d}), \]
  recalling that $l_{m + 1}^k = \cdots = l_{d - 1}^k = 0$. Thus, $k d + m + 1$
  is the largest occuring power of $A$ by which we divide, using homogeneity
  of the multiplier. The only non-zero entry not multiplied by a negative
  power of $A$ is now $l_m^k \neq 0$. When $m = j - 1$, this is by Taylor's
  theorem up to an error of magnitude $A^{- 1}$ equal to $- i \tmop{sign}
  (l_{j - 1}^k)$. When $m \neq j - 1$, this is up to an error of magnitude
  $A^{- 1}$ equal to 0.
  
  Writing
  \[ \vec{\theta}^s + \vec{A}^s \vec{\eta} = (\theta_0^s + A^{s d + 1} \eta_0,
     \ldots, \theta_{d - 1}^s + A^{s d + d} \eta_{d - 1}), \]
  we define the differences
  \begin{eqnarray*}
     D^j_{k d + m, A} (\vec{\theta}^0, \ldots, \vec{\theta}^k,
    \vec{\eta})
    & =&  \tilde{R}_{j, \vec{\eta}} \Phi^{\pm}_{k d + m} (\vec{\theta}^0 +
    \vec{A}^0 \vec{\eta}, \ldots, \vec{\theta}^k + \vec{A}^{k d + d}
    \vec{\eta}) \\
    &&- (\widetilde{H_j}^k \Phi^{\pm}_{k d + m}) (\vec{\theta}^0 +
    \vec{A}^0 \vec{\eta}, \ldots, \vec{\theta}^k + \vec{A}^{k d + d}
    \vec{\eta}) .
  \end{eqnarray*}
  In the first term, we apply the Riesz transforms $\tilde{R}_{j, \vec{\eta}}$
  in $\vec{\eta}$ while we apply $\widetilde{H_j}^k$ directly to
  $\Phi^{\pm}_{k d + m}$ and modulate the outcome. $\widetilde{H_j}^k
  \Phi^{\pm}_{k d + m}$ has multiplier $\tilde{m}_j^k (\vec{l}^k) =
  \tilde{m}_j (0, \ldots 0, l_m^k, 0, \ldots 0)$ with $l_m^k$ the last
  non-zero entry. From the above it follows that it differs by $A^{- 1}$ from
  the multiplier arising from the other term. We thus have for each $k$ and
  all $0 \leqslant m \leqslant d - 1$, a uniform estimate as follows: $\|
  D^j_{k d + m, A} (\vec{\theta}^0, \ldots, \vec{\theta}^k, \vec{\eta}) \|_X$
  is bounded by a multiple of finitely many terms of the form \ $\frac{1}{A} |
  \alpha^{\pm p}_{k d + m} (\vec{l}^0, \ldots, \vec{l}^k) | \| x_p \|_X$.
  
  We get in the usual way through successive integrations and using
  translation invariance:
  \begin{eqnarray}
    &  & \left| \mathbb{E}^{\vec{\theta}} \sum_{j = 1}^d \left\langle \sum_{k
    = 1}^M \sum_{m = 0}^{d - 1} \sum_{\sigma} \widetilde{H_j}^k
    \Phi^{\sigma}_{k d + m} (\vec{\theta}^0, \ldots, \vec{\theta}^k), \sum_{r
    = 1}^M \sum_{n = 0}^{d - 1} \sum_{\rho} \Gamma^{\rho}_{j, r d + n}
    (\vec{\theta}^0, \ldots, \vec{\theta}^r) \right\rangle_{{X, X^{\star}}^{}}
    \right| \nonumber\\
    & = & \left| \mathbb{E}^{\vec{\eta}} \mathbb{E}^{\vec{\theta}} \sum_{j =
    1}^d \left\langle \sum_k \sum_m \sum_{\sigma} \widetilde{H_j}^k
    \Phi^{\sigma}_{k d + m} (\vec{\theta}^0, \ldots, \vec{\theta}^k), \sum_{r
    = 1}^M \sum_{n = 0}^{d - 1} \sum_{\rho} \Gamma^{\rho}_{j, r d + n}
    (\vec{\theta}^0, \ldots, \vec{\theta}^r) \right\rangle_{{X, X^{\star}}^{}}
    \right| \nonumber\\
    & = & \left| \mathbb{E}^{\vec{\eta}} \mathbb{E}^{\vec{\theta}} \sum_{j =
    1}^d \langle \tmmathbf{\Phi}_{\vec{\theta}}^{\tilde{H}_j} (\vec{\eta}),
    \tmmathbf{\Gamma}_{j, \vec{\theta}} (\vec{\eta}) \rangle_{{X,
    X^{\star}}^{}} \right|  \label{last}
  \end{eqnarray}
  where
  \[ \tmmathbf{\Phi}_{\vec{\theta}}^{\tilde{H}_j} (\vec{\eta}) = \sum_{k =
     1}^M \sum_{m = 0}^{d - 1} \sum_{\sigma} \widetilde{H_j}^k
     \Phi^{\sigma}_{k d + m} (\vec{\theta}^0 + \vec{A}^0 \vec{\eta}, \ldots,
     \vec{\theta}^k + \vec{A}^k \vec{\eta}) \]
  and
  \[ \tmmathbf{\Gamma}_{j, \vec{\theta}} (\vec{\eta}) = \sum_{r = 1}^M
     \sum_{n = 0}^{d - 1} \sum_{\rho} \Gamma^{\rho}_{j, r d + n}
     (\vec{\theta}^0 + \vec{A}^0 \vec{\eta}, \ldots, \vec{\theta}^r +
     \vec{A}^r \vec{\eta}) . \]
  Writing
  \[ \tmmathbf{\Phi}_{\vec{\theta}}^{\tilde{R}_{j, \vec{\eta}}} (\vec{\eta}) =
     \sum_{k = 1}^M \sum_{m = 0}^{d - 1} \sum_{\sigma} \widetilde{R }_{j,
     \vec{\eta}} \Phi^{\sigma}_{k d + m} (\vec{\theta}^0 + \vec{A}^0
     \vec{\eta}, \ldots, \vec{\theta}^k + \vec{A}^k \vec{\eta}), \]
  we note that
  \begin{eqnarray*}
    &  & \left| \mathbb{E}^{\vec{\eta}} \mathbb{E}^{\vec{\theta}} \sum_{j =
    1}^d \langle \tmmathbf{\Phi}_{\vec{\theta}}^{\tilde{H}_j} (\vec{\eta})
    -\tmmathbf{\Phi}^{\tilde{R}_{j, \vec{\eta}}}_{\vec{\theta}} (\vec{\eta}),
    \tmmathbf{\Gamma}_{j, \vec{\theta}} (\vec{\eta}) \rangle_{{X,
    X^{\star}}^{}} \right|\\
    & \leqslant & \mathbb{E}^{\vec{\eta}} \mathbb{E}^{\vec{\theta}} \left(
    \sum_{j = 1}^d \| \tmmathbf{\Phi}_{\vec{\theta}}^{\tilde{H}_j}
    (\vec{\eta}) -\tmmathbf{\Phi}^{\tilde{R}_{j, \vec{\eta}}}_{\vec{\theta}}
    (\vec{\eta}) \|^2_X \right)^{1 / 2} \left( \sum_{j = 1}^d \|
    \tmmathbf{\Gamma}_{j, \vec{\theta}} (\vec{\eta}) \|^2_{{X^{\star}}^{}}
    \right)^{1 / 2}\\
    & \lesssim & \frac{1}{A} \| \vec{G} \|_{{L^q_{{X^{\star}}^{}}}  (\ell^2)}
    . 
  \end{eqnarray*}
  We may thus replace the expression in (\ref{last}) by
  \[ | \mathbb{E}^{\vec{\eta}} \mathbb{E}^{\vec{\theta}} \langle
     \tmmathbf{\Phi}_{\vec{\theta}}^{\tilde{R}_j, \vec{\eta}} (\vec{\eta}),
     \tmmathbf{\Gamma}_{j, \vec{\theta}} (\vec{\eta}) \rangle_{{X,
     X^{\star}}^{}} | \]
  at the cost of an error proportional to $A^{- 1}$. We thus inherit the bound
  from $\| \vec{R} \|_{L_X^p \rightarrow L_X^p (\ell^2)}$ and the proof of the
  lemma is complete.
  
  \ 
\end{proof}

Next, we have need to compare $\tilde{R}_j$ and $\mathcal{S}_j$ in a projected
form. Let us also define the operator $\pi_j$ on functions $f$ defined on
$\mathbb{T}^d$ by $\pi_j (f) (\theta_0, \ldots, \theta_{d - 1}) = \sum^1_{n =
- 2} \langle f (\theta_0, \ldots, \theta_{j - 1}, \cdot, \theta_j, \ldots,
\theta_{d - 1}) \rangle_{A_n}$, where $\langle \cdot \rangle_{A_n}$ means the
average and the arcs $A_n$ correspond to angles $[\nobracket n \pi / 2, (n +
1) \pi / 2) \nobracket$.

\begin{lemma}
  \label{lmHvsS}Denote by $\varphi_i^{\pm} (\theta_0, \ldots, \theta_{d - 1})
  = \varphi^{\pm} (\theta_{i - 1})$. There exists $c_0 > 0$ such that there
  holds
  \[ \pi_1 \tilde{R}_j \varphi_i^{\pm} = c_0 \mathcal{S}_j \varphi_i^{\pm} \]
  for all $0 \leqslant i \leqslant d - 1$.
\end{lemma}

\begin{proof}
  $\tilde{R}_j \varphi_i^{\pm} =\mathcal{S}_j \varphi_i^{\pm} = 0$ when $i
  \neq j - 1$. $\tilde{R}_j \varphi_{j - 1}^{\pm} =\mathcal{H}_j \varphi_{j -
  1}^{\pm}$. Now it suffices to use the corresponding fact for the Hilbert
  transform found in {\cite{DomPet2022a}}.
\end{proof}

\subsection{The proof of Theorem \ref{th1}}

\begin{proof}
  Let us first assume that $\langle f \rangle_{I_0} = \langle f, h_{I_0}
  \rangle = 0$. We estimate the norm of the dyadic Riesz vector by duality,
  using Lemma \ref{lmHvsS} in the last equality.
  \begin{eqnarray*}
    &  & \mathbb{E}^{\vec{\theta}} \sum_{j = 1}^d \left\langle \mathcal{S}_j
    \sum_{k = 1}^K \sum_{m = 0}^{d - 1} \sum_{\sigma} \mathd F^{\sigma}_{k d +
    m} (\vec{\theta}^0, \ldots, \vec{\theta}^k) \varphi^{\sigma} (\theta_m^k),
    \right.\\
    &  & \hspace{2em} \left. \sum_{r = 1}^R \sum_{n = 0}^{d - 1} \sum_{\rho} \mathd
    G^{\rho}_{j, r d + n} (\vec{\theta}^0, \ldots, \vec{\theta}^r)
    \varphi^{\rho} (\theta_n^r) \right\rangle_{X, X^{\star}}\\
    & = & \mathbb{E}^{\vec{\theta}} \sum_{j = 1}^d \left\langle \sum_{k =
    1}^K \sum_{m = 0}^{d - 1} \sum_{\sigma} \mathd F^{\sigma}_{k d + m}
    (\vec{\theta}^0, \ldots, \vec{\theta}^k) \mathcal{S}_j \varphi^{\sigma}
    (\theta_m^k), \right.\\
    &  & \hspace{2em}\left. \sum_{r = 1}^R \sum_{n = 0}^{d - 1} \sum_{\rho} \mathd
    G^{\rho}_{j, r d + n} (\vec{\theta}^0, \ldots, \vec{\theta}^r)
    \varphi^{\rho} (\theta_n^r) \right\rangle_{X, X^{\star}}\\
    & = & c_0^{- 1} \mathbb{E}^{\vec{\theta}} \sum_{j = 1}^d \left\langle
    \sum_{k = 1}^K \sum_{m = 0}^{d - 1} \sum_{\sigma} \mathd F^{\sigma}_{k d +
    m} (\vec{\theta}^0, \ldots, \vec{\theta}^k) \pi \tilde{R}_j
    \varphi^{\sigma} (\theta_m^k), \right.\\
    &  & \hspace{2em}\left. \sum_{r = 1}^R \sum_{n = 0}^{d - 1} \sum_{\rho} \mathd
    G^{\rho}_{j, r d + n} (\vec{\theta}^0, \ldots, \vec{\theta}^r)
    \varphi^{\rho} (\theta_n^r) \right\rangle_{X, X^{\star}} .
  \end{eqnarray*}
  Thanks to the identity (\ref{insideRiesz}), it suffices to compare this to
  \begin{eqnarray*} 
  &&c_0^{- 1} \mathbb{E}^{\vec{\theta}} \sum_{j = 1}^d \left\langle \sum_{k =
     1}^K \sum_{m = 0}^{d - 1} \sum_{\sigma} \mathd F^{\sigma}_{k d + m}
     (\vec{\theta}^0, \ldots, \vec{\theta}^k) \tilde{R}_j \varphi^{\sigma}
     (\theta_m^k), \right. \\
     &&\hspace{2em}
     \left. \sum_{r = 1}^R \sum_{n = 0}^{d - 1}
     \sum_{\rho} \mathd G^{\rho}_{j, r d + n} (\vec{\theta}^0, \ldots,
     \vec{\theta}^r) \varphi^{\rho} (\theta_n^r) \right\rangle_{X, X^{\star}}
  \end{eqnarray*}

  Since $\pi \tilde{R}_j \varphi^{\sigma} (\theta_m^k)$, $\tilde{R}_j
  \varphi^{\sigma} (\theta_m^k)$ and $\varphi^{\rho} (\theta_n^r)$ have mean
  zero, only diagonal terms with $k = r$ and $m = n$ yield non-zero
  contributions (and it is required that $m = j - 1$). But for diagonal terms,
  observe that for, say, $\theta = \theta_{j - 1}^k$,
  \[ \int_{\mathbb{T}} \pi \tilde{R}_j \varphi^{\sigma} (\theta)
     \varphi^{\rho} (\theta) \mathd \theta = \int_{\mathbb{T}} \tilde{R}_j
     \varphi^{\sigma} (\theta) \varphi^{\rho} (\theta) \mathd \theta \]
  because $\varphi^{\rho} (\theta)$ is constant on the quarter arcs. This
  suffices to see that the two expressions are the same. We apply the Lemma
  \ref{lmRmoved} to finish the estimate for the case $\langle f \rangle_{I_0}
  = \langle f, h_{I_0} \rangle = 0$. To pass to the general case, one uses the
  universal bound for averaging operators. 
\end{proof}
\[ \  \]

\end{document}